\documentclass[twoside]{amsart}
\usepackage{amsfonts,amsthm}
\usepackage{amsmath, amssymb, amsxtra}
\theoremstyle{plain}

\newtheorem*{theorem A}{Theorem~A}
\newtheorem*{theorem B}{Theorem~B}
\newtheorem{corollary}{Corollary}[section]
\newtheorem{lemma}{Lemma}[section]
\newtheorem{proposition}{Proposition}[section]
\newtheorem{remark}{Remark}[section]
\newtheorem*{mtheorem}{Main Theorem}

\newtheorem*{pro A}{Proposition~A}
\newtheorem*{corollary 01}{Corollary~1}
\newtheorem*{corollary 02}{Corollary~2}
\theoremstyle{plain}
\newcommand{\be}{\begin{equation}}
\newcommand{\ee}{\end{equation}}
\newcommand{\bea}{\begin{eqnarray}}
\newcommand{\eea}{\end{eqnarray}}
\newcommand{\ba}{\begin{array}}

\newcommand{\ea}{\end{array}}

\newcommand{\bc}{\begin{center}}
\newcommand{\ec}{\end{center}}

\newcommand{\benu}{\begin{enumerate}}
\newcommand{\eenu}{\end{enumerate}}
\newcommand{\bpr}{\begin{proposition}}
\newcommand{\epr}{\end{proposition}}
\newcommand{\ble}{\begin{lemma}}
\newcommand{\ele}{\end{lemma}}
\newcommand{\bco}{\begin{corollary}}
\newcommand{\eco}{\end{corollary}}

\def \NBo{SU_{2,m-1}/S(U_2{\cdot}U_{m-1})}
\def \NBt{SU_{2,m}/S(U_{2}{\cdot}U_{m})}
\def \CC{{\mathbb C}}

\def \RF{\it Reeb flow}
\begin{document}

\title[ Hypersurfaces in complex hyperbolic Grassmannians  ]{ Addendum to the paper ``Hypersurfaces with Isometric Reeb Flow in Complex hyperbolic Two-Plane Grassmannians"} \vspace{0.2in}
\author[H. Lee, M.J. Kim \& Y.J. Suh]{Hyunjin Lee, Mi Jung Kim and Young Jin  Suh}

\address{\newline
Young Jin Suh and Mi Jung Kim
\newline Department of Mathematics,
\newline Kyungpook National University,
\newline Taegu 702-701, KOREA}
\email{yjsuh@knu.ac.kr}

\address{\newline Hyunjin Lee
\newline The Center for Geometry and its Applications,
\newline Pohang University of Science \& Technology,
\newline Pohang 790-784, KOREA}
\email{lhjibis@hanmail.net}



\thanks{* The first author was supported by grant Proj. No. NRF-2012-R1A1A-3002031
and the third by grants Proj. Nos. NRF-2011-220-C00002 and NRF-2012-R1A2A2A-01043023 from National Research Foundation
of Korea.}

\begin{abstract}
We classify all of real hypersurfaces $M$ with Reeb invariant shape
operator in complex hyperbolic two-plane Grassmannians
$SU_{2,m}/S(U_2{\cdot}U_m)$, $m \ge 2$. Then it becomes a tube over
a totally geodesic $SU_{2,m-1}/S(U_2{\cdot}U_{m-1})$ in
$SU_{2,m}/S(U_2{\cdot}U_m)$ or a horosphere whose center at infinity
is singular and of type $JN{\in}{\mathfrak J}N$ for a unit normal vector field $N$ of $M$.
\end{abstract}

\maketitle

\section*{Introduction}
\setcounter{equation}{0}
\renewcommand{\theequation}{0.\arabic{equation}}
\vspace{0.13in}

 Let us introduce a paper due to Suh \cite{S3} for the classification of all real
hypersurfaces with isometric {\RF} in complex hyperbolic two-plane
Grassmann manifold $\NBt$ as follows:

\begin{theorem A}\label{Theorem A}
Let $M$ be a connected
orientable real hypersurface in complex hyperbolic two-plane Grassmannian $\NBt$, $m \geq 3$.
Then the Reeb flow on $M$ is isometric if and only if $M$ is an open
part of a tube around some totally geodesic $\NBo$ in $\NBt$ or a
horosphere whose center at infinity is singular and of type $JN{\in}{\mathfrak J}N$ for a unit normal vector field $N$ of $M$.
\end{theorem A}

A tube around  $\NBo$ in $\NBt$ is a principal orbit of the
isometric action of the maximal compact subgroup $SU_{1,m+1}$ of
$SU_{m+2}$, and the orbits of the {\RF} corresponding to the orbits
of the action of $U_1$. The action of $SU_{1,m+1}$ has two kinds of
singular orbits. One is a totally geodesic $\NBo$ in $\NBt$ and the
other is a totally geodesic $\CC H^m$ in $\NBt$.

\vskip 6pt

When the shape operator $A$ of $M$ in $\NBt$ is
Lie-parallel along the {\RF} of the Reeb vector field $\xi$, that is
${\mathcal L}_{\xi}A=0$, we say that the shape opearator is {\it Reeb invariant}. The purpose of this addendum is, by Theorem A, to give a
complete classification of real hypersurfaces in $\NBt$ with {\it Reeb invariant} shape operator as follows:

\begin{mtheorem}\label{Main Theorem}
Let $M$ be a connected orientable real hypersurface in complex hyperbolic two-plane Grassmannian $\NBt$, $m \geq 3$. Then the shape operator on $M$ is Reeb invariant
if and only if $M$ is an open part of a tube around some totally
geodesic $\NBo$ in $\NBt$ or a horosphere whose center at infinity
is singular and of type $JN{\in}{\mathfrak J}N$ for a unit normal vector field $N$ of $M$.
\end{mtheorem}

Moreover, related to the invariancy of shape operator, by using the result of Main Theorem, we have the following two corollaries.

\begin{corollary 01}\label{coro 1}
There does not exist any connected orientable real hypersurface in complex hyperbolic two-plane Grassmannian $\NBt$, $m \geq 3$, with $\mathcal F$-invariant shape operator.
\end{corollary 01}

\begin{corollary 02}\label{coro 2}
There does not exist any connected orientable real hypersurface in complex hyperbolic two-plane Grassmannian $\NBt$, $m \geq 3$, with invariant shape operator.
\end{corollary 02}

In previous corollaries, if the shape operator~$A$ of $M$ in $\NBt$ satisfies a property of $\mathcal L_{X}A =0$ on a distribution~$\mathcal F$ defined by $\mathcal F = \mathcal C^{\bot} \cup \mathcal Q^{\bot}$ (or for any tangent vector field $X$ on $M$, resp.), then it is said to be {\it $\mathcal F$-invariant} (or {\it invariant}, resp.).

\vskip 6pt

We use some references  \cite{BS}, \cite{BLS}, \cite{H1}, \cite{H2}, and \cite{MP} to recall the Riemannian geometry of complex hyperbolic two-plane Grassmannians $\NBt$. And some fundamental formulas related to the Codazzi and Gauss equations from the curvature tensor of complex hyperbolic two-plane Grassmannian $\NBt$ will be recalled (see \cite{PS}, \cite{MR}, and \cite{O}). In this addendum we give an important Proposition 1.1 and prove our Main Theorem in section~\ref{section 1}. Lastly, we give a brief proof for Corollaries~$1$ and $2$ by using our Main Theorem.

\vskip 17pt

\section  {Proof of the Main Theorem}\label{section 1}
\setcounter{equation}{0}
\renewcommand{\theequation}{1.\arabic{equation}}
\vspace{0.13in}

In order to give a complete proof of our Main Theorem in the
introduction, we need the following Key Proposition. Then by virtue of Theorem A we give a complete proof of our main theorem.
\par
\vskip 6pt

\begin{proposition}\label{Proposition 1.1} Let $M$ be a real hypersurface in noncompact complex two-plane Grassmannian $\NBt$, $m \geq 3$. If the shape operator on $M$ is Reeb invariant, then
the shape operator $A$ commutes with the structure tensor $\phi$.
\end{proposition}

\begin{proof} \quad First note that
\begin{equation*}
\begin{split}
(\mathcal{L}_{\xi}A)X & = {\mathcal L}_{\xi}(AX) - A \mathcal{L}_{\xi}X \\
& = \nabla_{\xi}(AX) - \nabla_{AX} \xi - A(\nabla_{\xi}X -
\nabla_{X}\xi)\\
& = (\nabla_{\xi}A)X - \nabla_{AX}\xi + A \nabla_{X}\xi\\
& = (\nabla_{\xi} A)X - \phi A^{2}X + A \phi AX
\end{split}
\end{equation*}
for any vector field $X$ on $M$. Then the assumption
$\mathcal{L}_{\xi}A=0$, that is, the shape operator is Reeb invariant if and only if $(\nabla_{\xi}A)X = \phi
A^{2} X - A \phi AX$.

\vskip 6pt

\noindent On the other hand, by the equation of Codazzi in \cite{S3} and the
assumption of Reeb invariant, we have
\begin{equation}\label{eq: 1(Fundamental equation)}
\begin{split}
(\nabla_{X}A) \xi = \phi & A^{2}X - A \phi AX \\
& + \frac{1}{2} \Big [
 \phi X + \sum_{\nu=1}^{3} \big\{ \eta_{\nu}(\xi) \phi_{\nu}X - \eta_{\nu}(X) \phi_{\nu} \xi + 3 \eta_{\nu}(\phi X)
 \xi_{\nu}\big\}\Big].
\end{split}
\end{equation}
Now, let us take an orthonormal basis $\{e_{1}, e_{2}, \cdots,
e_{4m-1}\}$ for the tangent space $T_{x}M$, $x \in M$, for $M$ in $\NBt$. Then the equation of Codazzi gives
\begin{equation}\label{eq: 2}
\begin{split}
(\nabla_{e_{i}}A)X - (\nabla_{X}A)e_{i}& = -\frac{1}{2}\Big
[\eta(e_{i})\phi X - \eta(X)\phi e_{i} - 2g(\phi e_{i},X)\xi \\
& \quad   + \sum_{\nu=1}^3 \big\{\eta_\nu(e_{i})\phi_\nu
X - \eta_\nu(X)\phi_\nu e_{i} - 2g(\phi_\nu e_{i},X)\xi_\nu\big\} \\
& \quad  + \sum_{\nu=1}^3 \big\{\eta_\nu(\phi
e_{i})\phi_\nu\phi X
-\eta_\nu(\phi X)\phi_\nu\phi e_{i}\big\} \\
& \quad + \sum_{\nu=1}^3 \big\{\eta(e_{i})\eta_\nu(\phi
X) - \eta(X)\eta_\nu(\phi e_{i})\big\}\xi_\nu \Big ],
\end{split}
\end{equation}
from which, together with the fundamental formulas mentioned in \cite{S3}, we know that
\begin{equation}\label{eq: 3}
\begin{split}
&\sum _{i=1}^{4m-1} g ((\nabla_{e_{i}}A)X, \phi e_{i}) \\
& \quad = (2  m-1)\eta(X) + \frac{1}{2} \sum_{\nu=1}^{3} \big\{ g(\phi_{\nu}\xi,
\phi_{\nu}X) - \eta_{\nu}(X) {\rm Tr} (\phi \phi_{\nu}) \big\} \\
& \quad \quad  + \frac{1}{2} \sum_{\nu=1}^{3} \big\{ g(
\phi_{\nu} \phi X, \phi \phi_{\nu}\xi) + \eta(X) g(\phi \xi_{\nu},
\phi \xi_{\nu}) \big\}\\
&=\, (2  m+1)  \eta(X) - \frac{1}{2} \sum_{\nu=1}^{3}
\eta_{\nu}(X){\rm Tr}\phi \phi_{\nu} -\sum_{\nu=1}^{3}
\eta_{\nu}(\xi) \eta_{\nu}(X),
\end{split}
\end{equation}
where the following formulas are used in the second equality
\begin{equation*}
\begin{array}{l}
\displaystyle \sum_{\nu=1}^{3}g(\phi_{\nu}\xi, \phi_{\nu}X)=
3\eta(X) - \sum_{\nu=1}^{3} \eta_{\nu}(\xi)
\eta_{\nu}(X),\\
\displaystyle \sum_{\nu=1}^{3}g(\phi_{\nu} \phi X, \phi \phi_{\nu
}\xi)= \sum_{\nu=1}^{3} \eta(X) \eta_{\nu}^{2}(\xi)-\sum_{\nu=1}^{3}
\eta_{\nu}(\xi) \eta_{\nu}(X),\\
\displaystyle \sum_{\nu=1}^{3} \eta(X) g(\phi\xi_{\nu}, \phi
\xi_{\nu })= 3\eta(X) - \sum_{\nu=1}^{3} \eta_{\nu}^{2}(\xi)\eta(X).
\end{array}
\end{equation*}

\vskip 3pt

\noindent Now let us denote by $U$ the vector $\nabla_{\xi}\xi= \phi
A \xi$. Then using the equation $(\nabla_{X}\phi)Y = \eta(Y)AX -
g(AX, Y) \xi$ given in \cite{S3} and taking the derivative to the vector field $U$ gives
\begin{equation*}
\nabla_{e_{i}}U = \eta(A\xi) Ae_{i} - g(Ae_{i}, A\xi)\xi + \phi
(\nabla_{e_{i}}A)\xi + \phi A (\nabla_{e_{i}}\xi).
\end{equation*}
Then naturally its divergence can be given by
\begin{equation}\label{eq: 4 (divergence of U)}
\begin{split}
{\rm div}\ U & = \sum_{i=1}^{4m-1} g(\nabla_{e_{i}} U, e_{i})\\
&= h \eta(A\xi) - \eta(A^{2}\xi) - \sum_{i=1}^{4m-1} g\big
((\nabla_{e_{i}}A)\xi, \phi e_{i} \big ) - \sum_{i=1}^{4m-1}g\big
(\phi Ae_{i}, A \phi e_{i}\big ),
\end{split}
\end{equation}
where $h$ denotes the trace of the shape operator of $M$ in $\NBt$.

\vskip 3pt

Now we calculate the squared norm of the tensor ${\phi}A-A{\phi}$ as follows:
\begin{equation}\label{eq: 5}
\begin{split}
\parallel  \phi A - A \phi \parallel ^{2} & = \sum_{i} g
\big( (\phi A - A \phi )e_{i},(\phi A - A \phi )e_{i} \big)\\
& =   \sum_{i, j} g \big( (\phi A - A \phi )e_{i},
e_{j}\big) g\big( (\phi A - A \phi )e_{i}, e_{j} \big)\\
& =   \sum_{i, j} \Big \{ g (\phi Ae_{j}, e_{i}) + g(\phi A e_{i},
e_{j})\Big \} \Big \{ g (\phi Ae_{j}, e_{i}) + g(\phi A e_{i},
e_{j})\Big \}\\
& =  2 \sum_{i,j} g(\phi Ae_{j}, e_{i}) g(\phi Ae_{j},
e_{i}) + 2 \sum_{i,j} g(\phi A e_{j}, e_{i})g(\phi Ae_{i}, e_{j})\\
& =   2\sum_{j} g(\phi A e_{j}, \phi A e_{j})  -2 \sum_{j} g(\phi A e_{j} , A \phi e_{j})\\
& = 2 {\rm div}\ U - 2h \eta(A\xi) + 2
\sum_{j}g\big((\nabla_{e_{j}}A)\xi, \phi e_{j}\big) + 2
\rm{Tr}A^{2},
\end{split}
\end{equation}
where $\sum_{i}$ (respectively, $\sum_{i, j}$) denotes the summation
from $i=1$ to $4m-1$ (respectively, from $i,j=1$ to $4m-1$) and in
the final equality we have used (\ref{eq: 4 (divergence of U)}).

\vskip 3pt

\noindent From this, together with the formula (\ref{eq: 3}), it
follows that
\begin{equation}\label{eq: 6}
\begin{split}
{\rm div}\ U =& \frac{1}{2} \parallel \phi A - A\phi \parallel^{2} -
{\rm Tr} A^{2}\\
& + h \eta(A\xi) - 2(m+1) + \frac{1}{2}\sum_{\nu=1}^{3}
\eta_{\nu}(\xi) \rm Tr \phi \phi_{\nu} + \sum_{\nu=1}^{3}
\eta_{\nu}^{2}(\xi).
\end{split}
\end{equation}
From \eqref{eq: 6}, together with the assumption of Reeb invariant shape operator, we want to show that the structure tensor
$\phi$ commutes with the shape operator $A$, that is, $\phi A = A
\phi $.

\vskip 3pt

Let us take the inner product (\ref{eq: 1(Fundamental
equation)}) with the Reeb vector field $\xi$. Then we have
\begin{equation}\label{eq: 7}
\begin{split}
g\big((\nabla_{X}A)\xi, \xi \big) & =-g(A \phi AX, \xi) +
\frac{1}{2}\sum_{\nu=1}^{3} \Big \{ \eta_{\nu}(\xi) g(\phi_{\nu}X,
\xi) + 3 \eta_{\nu}(\phi X)g(\xi_{\nu}, \xi )\Big\}\\
&= -g(A \phi AX, \xi) + 2 \sum_{\nu=1}^{3} \eta_{\nu}(\xi)\eta_{\nu}(\phi X)\\
&= g(AX, U) + 2 \sum_{\nu=1}^{3} \eta_{\nu}(\xi)
\eta_{\nu}(\phi X).
\end{split}
\end{equation}

\vskip 3pt

\noindent On the other hand, by applying the structure tensor $\phi$
to the vector field $U$, we have
\begin{equation*}
\phi U = \phi^{2}A\xi = -A\xi + \eta(A\xi)\xi = -A\xi + \alpha \xi,
\end{equation*}
where the function $\alpha$ denotes $\eta(A\xi)$. From this,
differentiating and using the formula $(\nabla_{X}\phi)Y = \eta(Y)AX
- g(AX, Y) \xi$ gives
\begin{equation}\label{eq: 8}
(\nabla_{X}A)\xi = g(AX, U) \xi - \phi (\nabla_{X}U) - A \phi AX +
(X\alpha) \xi + \alpha \phi AX.
\end{equation}
Taking the inner product (\ref{eq: 8}) with $\xi$ and using
$U= \phi A \xi$ gives
\begin{equation*}
\begin{split}
g\big( (\nabla_{X}A) \xi, \xi \big) & = g(AX, U) -g(A\phi AX, \xi) +
(X \alpha)\\
& = 2g(AX, U) +(X \alpha).
\end{split}
\end{equation*}
Then, together with (\ref{eq: 7}), it follows that
\begin{equation}\label{eq: 9}
g(AX, U) -2 \sum_{\nu=1}^{3}\eta_{\nu}(\xi) \eta_{\nu}(\phi X) +
(X\alpha) = 0.
\end{equation}
Substituting (\ref{eq: 1(Fundamental equation)}) and (\ref{eq:
9}) into (\ref{eq: 8}) gives
\begin{equation}\label{eq: 10}
\begin{split}
2\sum_{\nu=1}^{3} & \eta_{\nu}(\xi)\eta_{\nu}(\phi X) \xi - \phi
(\nabla_{X}U)\\
& =   \frac{1}{2}\phi X +\phi A^{2}X - \alpha \phi AX \\
& \quad +\frac{1}{2}\sum_{\nu=1}^{3}\Big\{
\eta_{\nu}(\xi) \phi_{\nu}X - \eta_{\nu}(X) \phi_{\nu}\xi + 3
\eta_{\nu}(\phi X)
\xi_{\nu}\Big\}.
\end{split}
\end{equation}
Then the above equation can be arranged as follows\,:
\begin{equation*}
\begin{split}
\phi \nabla_{X}U & = - \frac{1}{2} \phi X -\phi A^{2}X + \alpha \phi
AX + 2\sum_{\nu=1}^{3} \eta_{\nu}(\xi)
\eta_{\nu}(\phi X) \xi  \\
& \quad + \frac{1}{2} \sum_{\nu=1}^{3}\Big\{ \eta_{\nu}(X)
\phi_{\nu}\xi - \eta_{\nu} (\xi) \phi_{\nu} X - 3 \eta_{\nu}(\phi X)
\xi_{\nu} \Big\}.
\end{split}
\end{equation*}
From this, summing up from $1$ to $4m-1$ for an orthonormal basis of
$T_{x}M$, $x \in M$, we have
\begin{equation}\label{eq: 11}
\begin{split}
& \quad \sum_{i} g (\phi \nabla_{e_{i}}U, \phi e_{i}) = {\rm div}\, U +
\parallel U \parallel^{2}\\
& = - \frac{1}{2} \sum_{i} g(\phi e_{i}, \phi e_{i}) -\sum_{i}g(\phi
A^{2}e_{i}, \phi e_{i}) + \alpha \sum_{i } g( \phi Ae_{i}, \phi
e_{i}) \\
& \quad + \frac{1}{2} \sum_{\nu=1}^{3} \sum_{i} \Big\{
\eta_{\nu}(e_{i})g(\phi_{\nu}\xi, \phi e_{i}) - \eta_{\nu} (\xi)
g(\phi_{\nu} e_{i}, \phi e_{i}) - 3 \eta_{\nu}(\phi e_{i}) g(\phi e_{i}, \xi_{\nu}) \Big\}\\
& = -2(m+1)- {\rm Tr} A^{2}+ \eta(A^{2}\xi)+ \alpha h - \alpha ^{2} + \sum_{\nu=1}^{3} \eta_{\nu}^{2}(\xi) +
\frac{1}{2}\sum_{\nu=1}^{3}\eta_{\nu}(\xi) {\rm Tr}(\phi \phi_{\nu}),
\end{split}
\end{equation}
where in the first equality we have used the notion of ${\rm div}\ U$. Then it follows that
\begin{equation}\label{eq: 12}
{\rm div}\ U = -2(m+1)- {\rm Tr} A^{2}+ \alpha h + \sum_{\nu=1}^{3}
\eta_{\nu}^{2}(\xi) + \frac{1}{2}\sum_{\nu=1}^{3}\eta_{\nu}(\xi)
{\rm Tr}(\phi \phi_{\nu}),
\end{equation}
where we have used $\parallel U \parallel^{2}= \eta(A^{2}\xi) -
\alpha^{2}$ in (\ref{eq: 11}).

\vskip 6pt
\par
Now if we compare (\ref{eq: 6}) with the formula (\ref{eq: 12}), we
can assert that the squared norm $\parallel \phi A - A \phi
\parallel^{2}$ vanishes, that is, the structure tensor $\phi$ commutes with the shape operator $A$. This completes the
proof of our proposition.
\end{proof}
\vskip 6pt
\par
Hence by Proposition~\ref{Proposition 1.1} we know that the Reeb flow on $M$ is isometric. From this, together with Theorem~A we give a complete
proof of our Main Theorem in the introduction.
\qed
\vskip 6pt
\par
\begin{remark} It can be easily checked that the shape operator of real hypersurfaces $M$ in $\NBt$ is Reeb invariant, that is, ${\mathcal L}_{\xi}A=0$ when $M$ is locally congruent to a tube around some totally
geodesic $\NBo$ in $\NBt$ or a horosphere whose center at infinity
is singular and of type $JN{\in}{\mathfrak J}N$ for a unit normal vector field $N$ of $M$. So the converse of our main theorem naturally holds.
\end{remark}

\vskip 15pt

\section{Proof of Corollaries}\label{section 2}
\setcounter{equation}{0}
\renewcommand{\theequation}{1.\arabic{equation}}
\vspace{0.13in}

From the definitions of three kinds of the invariancy of the shape operator $A$ defined on $M$ in the Introduction, namely {\it invariant}, $\mathcal F$-{\it invariant} and {\it Reeb invariant} shape operator,  the notion of {\it Reeb invariant} is the most weakest condition. Thus from our Main Theorem, we assert that {\it if a real hypersurface~$M$ in $\NBt$, $m \geq 3$, has $\mathcal F$-invariant (or invariant) shape operator, then $M$ is locally congruent to a tube around some totally geodesic $\NBo$ in $\NBt$ or a horosphere whose center at infinity is singular}.

\vskip 3pt

Conversely, if we check whether a tube~$M_{r}$ of radius $r$ around the totally geodesic $\NBo$ in $\NBt$ and a horosphere~$\mathcal H$ in $\NBt$ whose center at infinity is singular have the $\mathcal F$-invariant (or invariant) shape operator, then it does not hold. In fact, we get a contradiction for the case $(\mathcal L_{\xi_{2}}A)\xi_{3}$. From such a view point, we can assert that {\it the shape operator~$A$ of~$M_{r}$ (or $\mathcal H$, respectively) satisfy neither the property of $\mathcal F$-invariant nor invariant shape operator.}

\vskip 3pt

Summing up these discussion, we give a complete proof of our Corollaries in the introduction. \qed

\end{document}